\documentclass[12pt]{article}
\usepackage{amssymb,amsmath,amsthm}
\usepackage[dvips]{graphics}

\newcount\mwhr\newcount\mwmin
\mwhr=\number\time\divide
\mwhr by60 \mwmin=\mwhr\multiply\mwmin by-60
\advance\mwmin by\number\time{}
\newtheorem{theorem}{Theorem}
\newtheorem{lemma}[theorem]{Lemma}

\newtheorem{corollary}[theorem]{Corollary}

\newcommand\eps{\varepsilon}
\newcommand\R{\mathbb{R}}
\newcommand\Z{\mathbb{Z}}
\newcommand\E{\mathbb{E}}
\newcommand\Prb{\mathbb{P}}
\newcommand\Po{\mathrm{Po}}

\newcommand\cF{\mathcal{F}}
\newcommand\cG{\mathcal{G}}
\newcommand\cP{\mathcal{P}}
\newcommand\cY{\mathcal{Y}}
\newcommand\tsum{{\textstyle\sum}}
\newcommand\Fg[1]{Figure~\ref{f:#1}}
\newcommand\Lm[1]{Lemma~\ref{l:#1}}
\newcommand\Co[1]{Corollary~\ref{c:#1}}
\newcommand\Th[1]{Theorem~\ref{t:#1}}

\begin{document}

\title{A critical constant for the $k$-nearest neighbour model}
\author{Paul Balister%
\thanks{University of Memphis, Department of
Mathematics, 3725 Norriswood, Memphis, TN 38152, USA}
\and B\'ela Bollob\'as$^*$%
\thanks{Trinity College, Cambridge CB2 1TQ, UK}
\and Amites Sarkar$^*$%
\and Mark Walters%
\thanks{Peterhouse, Cambridge CB2 1RD}}

\maketitle

\begin{abstract}
Let $\cP$ be a Poisson process of intensity one in a square $S_n$ of area $n$.
For a fixed integer $k$, join every point of $\cP$ to its $k$ nearest neighbours,
creating an undirected random geometric graph $G_{n,k}$. We prove that there exists
a critical constant $c_{\text{crit}}$ such that for $c<c_{\text{crit}}$,
$G_{n,\lfloor c\log n\rfloor}$ is disconnected with probability tending to 1 as
$n\to\infty$, and for $c>c_{\text{crit}}$,
$G_{n,\lfloor c\log n\rfloor}$ is connected with probability tending to 1 as
$n\to\infty$. This answers a question posed by the authors in~\cite{BBSW}.
\end{abstract}

Let $\cP$ be a Poisson process of intensity one in a square $S_n$ of area $n$.
For a fixed integer $k$, we join every point of $\cP$ to its $k$ nearest neighbours,
creating an undirected random geometric graph $G_{S_n,k}=G_{n,k}$ in which every vertex
has degree at least $k$. The connectivity of these graphs was studied by the present
authors in~\cite{BBSW}. It is not hard to see that $G_{n,k}$ becomes connected
around $k=\Theta(\log n)$, and we proved in~\cite{BBSW} that if $k(n)\le 0.3043\log n$
then the probability that $G_{n,k(n)}$ is connected tends to zero as $n\to\infty$,
while if $k(n)\ge 0.5139\log n$ then the probability that $G_{n,k(n)}$ is connected
tends to one as $n\to\infty$. However, we were unable to prove the natural conjecture
that there exists a critical constant $c_{\text{crit}}$ such that for
$c<c_{\text{crit}}$,
\[
\Prb(G_{n,\lfloor c\log n\rfloor}{\rm \ is\ connected})\to 0
\]
and for $c>c_{\text{crit}}$,
\[
\Prb(G_{n,\lfloor c\log n\rfloor}{\rm \ is\ connected})\to 1
\]
as $n\to\infty$. In this paper we prove this conjecture.

Central to the proof is the observation that, while there are no isolated vertices in
$G_{n,k}$, the obstructions to connectivity are nonetheless {\em small}. More precisely,
we have the following lemma, which is immediate from the proofs of Lemmas 2 and~6
of~\cite{BBSW}.

\begin{lemma}\label{l:small}
 For fixed\/ $c>0$ and\/ $L$, there exists $c'=c'(c,L)>0$, depending only on $c$
 and\/~$L$, such that for any $k\ge c\log n$, the probability that\/ $G_{n,k}$ contains
 two components each of\/ $($Euclidean$)$ diameter at least\/ $c'\sqrt{\log n}$, or
 any edge of length at least\/ $c'\sqrt{\log n}$, is $O(n^{-L})$.
\end{lemma}

This lemma enables us to restrict attention to ``local'' events, whose probabilities
we will estimate. Although heuristics and numerical evidence suggest that the actual
obstructions to connectivity arise far from the boundary of $S_n$, we were unable to
prove this in~\cite{BBSW}. Therefore we must consider the following two pairs of families
of events.

Let $M$ be a large integer, which we will choose in a moment. For the first pair,
we consider a Poisson process $\cP_S$ of intensity one in the square
$S=[-\frac12M\sqrt k,\frac12M\sqrt k]^2$ of area $M^2k$
centred at the origin, and construct the random graph $G_{S,k}=G_{M^2k,k}$ as above.
The event $A_k$ occurs when $G_{S,k}$ contains a component all of whose vertices lie within
the central square $S'=\tfrac{1}{2}S=\{\frac{x}{2}:x\in S\}$ of area $\frac14 M^2k$, and the event
$A'_k$ occurs when $G_{S,k}$ contains a component all of whose vertices lie within the central
square $S''=\tfrac{3}{4}S=\{\frac{3x}{4}:x\in S\}$ of area $\frac{9}{16}M^2k$.

For the second family, let $\cP_R$ be a Poisson process of intensity one in the square
$R=[0,M\sqrt k]\times[-\frac12M\sqrt k,\frac12M\sqrt k]$ of area $M^2k$, and join
every point of $\cP_R$ to its $k$ nearest neighbours to form the random geometric graph
$G_{R,k}$. The event $B_k$ occurs when $G_{R,k}$ contains a component all of whose
vertices lie within the square $R'=\tfrac{1}{2}R$, and the
event $B'_k$ occurs when $G_{R,k}$ contains a component all of whose vertices lie within the
rectangle $R''=\tfrac{3}{4}R$ (see \Fg{0}).

\begin{figure}
\centerline{\includegraphics{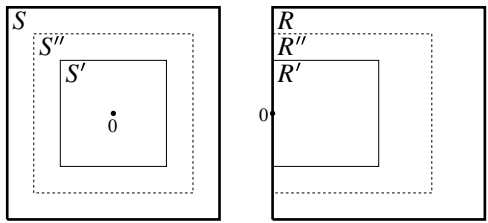}}
\caption{Regions used in defining $A_k$, $A'_k$, $B_k$, $B'_k$.}
\label{f:0}
\end{figure}

We now discuss the choice of $M$. It should be large enough to ensure that the probability of
seeing a long edge or two large components (relative to the size of $S$ or $R$) is much smaller
than the probabilities of the four events above. Specifically, we shall choose $M$ so that $M\ge40$ and
\begin{equation}\label{e:m1}
 \Prb(G_{n,k}\text{ contains two components with diameter greater than }
 \tfrac{1}{8}M\sqrt{k})=o(e^{-9k})
\end{equation}
(see \Lm{limsup} and \Co{edges}). Now we may assume, from the results in~\cite{BBSW}, that
$0.30\log n<k<0.52\log n$, so that
\[
 n^{-5}=o(e^{-9k})\qquad\text{and}\qquad
 \tfrac{1}{8}\sqrt{k}>\tfrac{1}{15}\sqrt{\log n}.
\]
Therefore, using the notation of \Lm{small}, it will be enough to take
\[
  M=\max\{15c'(0.3,5),40\}.
\]
From now on, no more reference will be made to the choice of $M$.

Our first target is to estimate $p_1(k)=\Prb(A_k)$ and $p_2(k)=\Prb(B_k)$. Specifically,
we will show that
\[
p_1(k)=e^{-(c_1+o_k(1))k}\quad\text{and}\quad
\max(p_1(k),p_2(k))=e^{-(c_2+o_k(1))k}.
\]
Defining
\[
f_1(k)=-\frac{\log p_1(k)}{k}\quad\text{and}\quad f_2(k)=-\frac{\log p_2(k)}{k},
\]
we will prove the following.

\begin{theorem}\label{t:akbk}
\[
 c_1=\lim_{k\to\infty}f_1(k)\quad\text{and}\quad c_2=\lim_{k\to\infty}\min\{f_1(k),f_2(k)\}
 \quad\text{exist.}
\]
\end{theorem}

The proof of this theorem, given in the next section, will occupy most of the paper.
Having established it, two straightforward tiling arguments will complete the
proof of the conjecture. The main idea in the proof of \Th{akbk} is that, for a fixed
$\eps>0$, there is a decomposition of the probability space of $G_{S,k}$ (or $G_{R,k}$)
into a finite set $\cF(\eps)$ of disjoint events or {\em configurations}, such
that the knowledge of which configuration occurs almost always determines
``up to $\eps$'' whether or not $A_k$ (or $B_k$) occurs. Once we have
this set of configurations, we can accurately estimate the probability of each one using
the following lemma, which is Lemma~1 of~\cite{BBSW}. (The proof of the lemma is just
a simple computation.)

\begin{lemma}\label{l:plogp}
 Let\/ $A_1,\dots,A_r$ be disjoint regions of\/ $\R^2$ and\/
 $\rho_1,\dots,\rho_r\ge0$ real numbers such that\/ $\rho_i|A_i|\in\Z$.
 Then the probability that a Poisson process with intensity~$1$ has
 precisely $\rho_i|A_i|$ points in each region $A_i$ is
 \[
  \exp\left\{\sum_{i=1}^r(\rho_i-1-\rho_i\log\rho_i)|A_i|
  +O(r\log_+\tsum\rho_i|A_i|)\right\}
 \]
 with the convention that\/ $0\log 0=0$, and\/ $\log_+ x=\max(\log x,1)$.
\end{lemma}

One of the configurations for which $A_k$ (or $B_k$) occurs will dominate,
in the sense that it will have the highest probability of all such configurations,
and we will be able to read off the value of $c_1$ (or $c_2$) from it.

\section{Proof of Theorem 2}

Let us fix $k$ and estimate $p_1(k)=\Prb(A_k)$ and $p_2(k)=\Prb(B_k)$. We will consider
very fine discretizations of the square regions $R$ and $S$ (both of area $M^2k$).
In the following, we will frequently have to neglect certain ``bad" events.
We must show that the probability of each of these events is negligible compared to those
of $A_k$ and $B_k$. For this we will need lower bounds on $p_1(k)$ and $p_2(k)$, or, more
precisely, upper bounds on $\limsup_{k\to\infty}f_1(k)$ and $\limsup_{k\to\infty}f_2(k)$.
Such bounds are provided below. We follow the method of~\cite{BBSW}, although a version
of this lemma (with larger constants) was obtained earlier by Xue and Kumar~\cite{XuKu}.

\begin{figure}
\centerline{\includegraphics{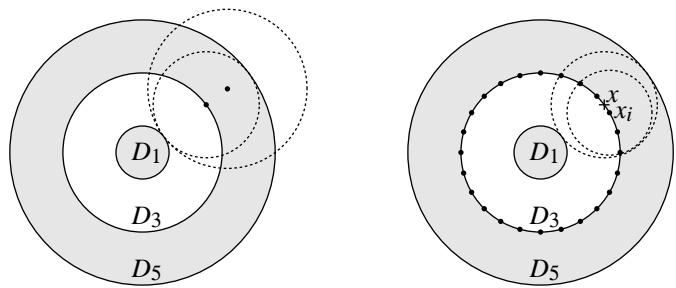}}
\caption{The regions $D_1$, $D_3$ and $D_5$ used in the proof of \Lm{limsup}.}
\label{f:1}
\end{figure}

\begin{lemma}\label{l:limsup}
 \[
  \limsup_{k\to\infty} f_1(k)\le 8\qquad\text{and}\qquad
  \limsup_{k\to\infty} f_2(k)\le 8.
 \]
\end{lemma}
\begin{proof}
Consider a configuration of three concentric discs $D_1$, $D_3$ and $D_5$,
of radii $r$, $3r$ and $5r$ respectively, where $\pi r^2=k+1$ (see \Fg{1}). Since
the diameter of $D_5$ is at most $8\sqrt k$ and $M\ge40$, one can choose the centre of
the discs so that all the discs lie entirely within the central square $S'$ (or $R'$).
Call the configuration
{\em bad\/} if (I) $D_1$ contains at least $k+1$ points, (II) the annulus $D_3\setminus D_1$
contains no points, and (III) the intersection of $D_5\setminus D_3$ with any disc of radius
$2r$ centred at a point $P$ on the boundary of $D_3$ contains at least $k+1$ points. Now if
the configuration is bad, then $A_k$ (or $B_k$) will occur, because the $k$ nearest neighbours of a
point in $D_1$ all lie within $D_1$ and the $k$ nearest neighbours of a point outside $D_3$
all lie outside $D_3$. (Otherwise, there would be a point $x$ outside $D_3$ and a disc
centred at $x$ touching $D_1$ that contained fewer than $k+1$ points. But this disc
contains a disc of radius $2r$ about some point on the boundary of $D_3$, contradicting (III).)
Hence there will be no edge connecting a point inside $D_1$ to a point outside $D_1$.
Condition (I) holds with probability about $\frac12$ (in fact, slightly more than $\frac12$),
and condition (II) holds with probability
$e^{-8(k+1)}$. Now consider Condition (III). Note that
there is an $\eps>0$ such that any disc of radius $(2-\eps)r$ around any point $x$ on the
boundary of $D_3$ intersects the annulus $D_5\setminus D_3$ in a region $D_x$ of area $2(k+1)$. It
follows from the concentration of the Poisson distribution (see for instance \Lm{bound}) that
the probability that $D_x$ contains less than $k+1$ points is $o_k(1)$.
Pick points $x_1,\dots, x_t$ around the boundary of $D_3$ so that any point of the boundary
of $D_3$ is within $\eps r$ of some $x_i$. Clearly we can choose $t=\lceil 3\pi/\eps\rceil$,
so that $t$ is independent of~$k$. Hence the probability that any $D_{x_i}$ contains fewer
than $k+1$ points is $o_k(1)$, but any disc of radius $2r$ about $x$ contains a disc of
radius $(2-\eps)r$ about some $x_i$. Thus the probability that any such $x$ exists with
the disc of radius $2r$ about $x$ containing fewer than $k+1$ points is $o_k(t)=o_k(1)$,
and so Condition (III) holds with probability $1-o_k(1)$.
Since the events corresponding to conditions (I), (II) and (III) are independent,
$p_1(k),p_2(k)\ge e^{-(8+o_k(1))k}$ and the result follows.
\end{proof}

Recall that in the last section we defined four families of events $A_k$, $A'_k$, $B_k$ and $B'_k$.
We are only really interested in $A_k$ and $B_k$; the events $A'_k$ and $B'_k$ arise only
because of a technicality, and it will be convenient to prove a simple lemma (\Lm{tech})
about them at the outset. Before we do this, it will be convenient to prove a simple
lemma bounding the Poisson distribution, and deduce a bound on the edge lengths in $G_{S,k}$.

\begin{lemma}\label{l:bound}
If $\rho>1$ then
\[
\Prb({\rm Po}(A)\ge\rho A)\le e^{(\rho-1-\rho\log\rho)A}.
\]
If $\rho<1$ then
\[
\Prb({\rm Po}(A)\le\rho A)\le e^{(\rho-1-\rho\log\rho)A}.
\]
\end{lemma}
\begin{proof}
Let $X\sim{\rm Po}(A)$. Then
\[
\E(\rho^X)=\sum_{n=0}^{\infty}\rho^n\frac{A^n}{n!}e^{-A}=e^{(\rho-1)A}.
\]
Therefore if $\rho>1$ then
\[
\Prb(X>\rho A)\le \E(\rho^{X-\rho A})=e^{(\rho-1-\rho\log\rho)A},
\]
and if $\rho<1$ then
\[
\Prb(X<\rho A)\le \E(\rho^{X-\rho A})=e^{(\rho-1-\rho\log\rho)A}.
\]
\end{proof}

\begin{corollary}\label{c:edges}
 For any $m$ with $M^2k\le m\le n$ and $0.3\log n\le k$,
 the probability that\/ $G_{m,k}$ contains an
 edge of length at least $\frac18M\sqrt k$ is $o(e^{-9k})$.
\end{corollary}
Note that this does not quite follow from \Lm{small}, since reducing
the area of the square, and hence the number of vertices, could in
principle increase the number of long edges in the remaining graph.
\begin{proof}
If some vertex $v$ of $G_{m,k}$ has its $k^{\rm th}$ nearest neighbour at a distance
more than $\frac18M\sqrt k\ge 5\sqrt k$, then there must be fewer than $k$
points within a quarter-disc of area $\frac{\pi}{4}25k>19k$ inside $S_m$.
(We need to consider quarter-discs since $v$ may be close to a corner of $S_m$.
The lower bound $M^2k\le m$ ensures that the quarter-disc fits.)
By \Lm{bound}, this occurs with probability at most
$e^{(1/19-1-(1/19)\log(1/19))19k}<e^{-15k}$. The expected number of vertices
where this will occur is thus $O(me^{-15k})=o(e^{-9k})$ since $m\le n\le e^{k/0.3}$.
Thus the probability that $G_{m,k}$ contains an
edge of length at least $\frac18M\sqrt k$ is $o(e^{-9k})$.
\end{proof}

\begin{lemma}\label{l:tech}
\[\Prb(A_k)\le \Prb(A'_k)\le (4+o_k(1))\Prb(A_k),\]
\[\Prb(B_k)\le \Prb(B'_k)\le (2+o_k(1))(\Prb(A_k)+\Prb(B_k)).\]
\end{lemma}
\begin{proof}
Both lower bounds are immediate. For the first upper bound, fix a Poisson process with
intensity 1 in the square $S_n$ of area $n$ centred at the origin.
Let $T$ be the square of side length $\frac54M\sqrt k$, also centred at the origin.
Note that for sufficiently large $k$ and $0.3\log n\le k\le 0.52\log n$, $T\subseteq S_n$,
so we shall assume this in the following.

Cover $T$ with four translates $S_1,\dots,S_4$ of $S$ as shown in \Fg{2}.
We now define three ``bad'' events. Let $E_1$ be the event that $G_{n,k}$ contains two
components of diameter greater than $\frac18M\sqrt k$. By \eqref{e:m1} we know that
$\Prb(E_1)=o(e^{-9k})$. Let $E_2$ be the event that some edge in either $G_{n,k}$
or in one of the $G_{S_i,k}$ is of length greater than $\frac18M\sqrt k$.
By \Co{edges}, $\Prb(E_2)=o(e^{-9k})$. Finally, let $E_3$ be the event that there
is no component in $G_{n,k}$ with at least one vertex outside of $T$ and with
diameter greater than $\frac18M\sqrt k$. Note that if we divide some square $\tilde S$
in $S_n$ of area $M^2k$ into $(8M)^2$ small squares, each of side length $\frac18\sqrt k$,
then with probability bounded away from zero (independently of $k$), there will be at
least one, and at most $\frac{k}{21}$ vertices in each small square. But then it is
easy to see that every vertex in a small square is adjacent in $G_{n,k}$
to every vertex in any neighbouring small square, provided these squares are at least
distance $\frac{3}{8}\sqrt k$ from the boundary of $\tilde S$ (see \Fg{3}). In
this case, there will be a large component of $G_{n,k}$ intersecting $\tilde S$.
Since we can place $\Omega(n/k)=\omega(k)$ independent copies of $\tilde S$ in $S_n$,
all avoiding $T$, we see that $\Prb(E_3)=e^{-\omega(k)}$. In particular,
$\Prb(E_3)=o(e^{-9k})$.

Assume the event $A'_k$ occurs, i.e., there is a small component $C$ of $G_{S,k}$ inside
$S''=\frac{3}{4}S$. Assume also that $E=E_1\cup E_2\cup E_3$ does not hold.
Then $C$ must also be a component (or a union of components) in $G_{n,k}$,
since the addition of vertices outside of $S$ will not cause any new edge to form
within $S$, and no vertex outside of $S$ can be joined to a vertex in $S''$,
since this edge would be of length greater than $\frac{1}{8}M\sqrt k$ in $G_{n,k}$.
Since $E_3$ and $E_1$ do not hold, there is no component of $G_{n,k}$
of diameter greater than $\frac18M\sqrt k$ entirely within $T$. Thus $C$ is of diameter
at most $\frac{1}{8}M\sqrt k$. Since $C$ lies
inside $S''$, it must lie entirely within at least one of the four translates $S'_i$ of
$S'$ corresponding to the $S_i$. (For example, if $C$ contains any vertex in the
top left quadrant of $S''$, then the whole component must lie in $S'_1$ in \Fg{2}.)
No edge occurs in $E(G_{S_i,k})\setminus E(G_{n,k})$ between vertices
within $S''_i$, since otherwise there would be an edge from a vertex in $S''_i$
to $S_n\setminus S_i$ in $G_{n,k}$ of length greater than $\frac18M\sqrt k$.
Since no edge of $G_{S_i,k}$ is longer than $\frac18M\sqrt k$, no such edge
joins a vertex in $S'_i$ to a vertex outside $S''_i$. Thus $C$ remains a component
in $G_{S_i,k}$ and lies entirely within $S'_i$. Hence one of the events $A_k$
corresponding to the four copies $S_i$ of $S$ occurs. Thus
$\Prb(A'_k\setminus E)\le 4\Prb(A_k)$ and so $\Prb(A'_k)\le 4\Prb(A_k)+\Prb(E)$.
But $\Prb(E)=o(e^{-9k})$, so by \Lm{limsup}, $\Prb(A'_k)\le(4+o_k(1))\Prb(A_k)$.

The upper bound for $\Prb(B'_k)$ is similar. In this case, the squares $T$ and $S_n$
are both aligned so as to share part of their leftmost boundaries with $R$ (see \Fg{2}).
The region $R''$ is covered by four central squares $R'_1$, $R'_2$, $S'_1$, and $S'_2$,
of the four squares $R_1$, $R_2$, $S_1$, and $S_2$, all of which lie in $T$.
There are two possibilities. Either our small component $C$ in
$R''$ lies in the left half of $R$, and hence in one of the $R_i'$, an event which has
probability at most $(2+o_k(1))\Prb(B_k)$ by an argument similar to the one above.
The other possibility is that the small component strays
into the right half of $R$, and so lies in one of the $S_i'$, an event with probability
at most $(2+o_k(1))\Prb(A_k)$. This proves the lemma.
\end{proof}

\begin{figure}
\centerline{\includegraphics{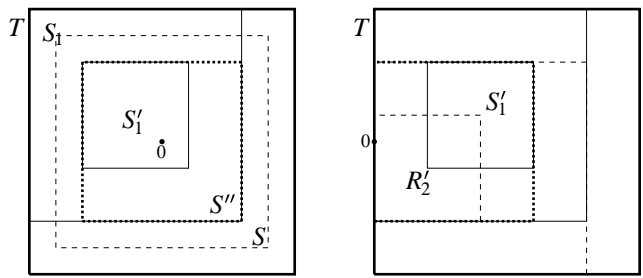}}
\caption{Left: Square $T$ is covered by squares $S_1,\dots,S_4$ aligned to the four
corners of $T$ (solid thin line, only $S_1$ shown). The
smaller squares $S'_i$ (solid thin line) then cover $S''$ (dotted line). The square $S$
(dashed line) is also shown. Right: corresponding picture for $B'_k$, with $R''$
(dotted line) covered by squares $S'_1,S'_2,R'_1,R'_2$ (only $S'_1$ and $R'_2$ shown).}
\label{f:2}
\end{figure}

Now we will restrict attention to $A_k$, $p_1(k)$ and $f_1(k)$.
Fix $0<\eps<\tfrac12$ and $M$ and choose $N=N(\eps,M)\gg M^2/\eps$.
We will consider $k\gg M,N$ to be a very large fixed integer. Now tile the
$M\sqrt{k}\times M\sqrt{k}$ square $S$, centred at 0, with $(MN)^2$ squares of
side length $\ell=\sqrt{k}/N$ and hence area $\ell^2=k/N^2$.

Next we wish to define a {\em configuration}. For a fixed instance of $\cP_S$, we label
each small square $S_i$ with the {\em approximate density\/} $d(S_i)$ of points in $S_i$,
where $d(S_i)$ is defined precisely by the formula
\begin{equation}\label{dSi}
d(S_i)=
\begin{cases}
0&\text{if $S_i$ contains no points of $\cP_A$}\\
\frac{\lceil N^3r/k\rceil}{N}&\text{if $S_i$ contains $r$ points of $\cP_A$, where $r\le k$}\\
\infty&\text{if $S_i$ contains more than $k$ points of $\cP_A$.}\\
\end{cases}
\end{equation}
We call such a labelled square $S$ a {\em configuration} $F$, and we say that $\cP_S$ has
(or belongs to) type $F$. Note that the total number of configurations is exactly
\[
(N^3+2)^{(MN)^2}.
\]
The aim is that the configuration $F$ should contain enough information about $\cP_S$
to determine whether or not $A_k$ occurs up to a small error, while the set of
all possible configurations is nevertheless finite.

The next step is to identify a set of undesirable, or {\em bad}, configurations, and
discard them. Of course, we are really discarding all instances of $\cP_S$ which belong
to a bad configuration, but we will think of discarding the configurations themselves,
and speak, for instance, of the measure of a set $\cF$ of configurations when
we mean the probability that $\cP_S$ belongs to some $F\in\cF$.

For an instance $\cP_S$ of the Poisson process in $S$, let $F(\cP_S)$ be the
configuration it belongs to. There will be two types of bad configuration in total.

\noindent{\bf Type A.} These are configurations which contain a square $S_i$ with
$d(S_i)>N^2/21$. (We may assume that 21 divides $N$ so that $N^2/21$ is an integer.)
In this case $S_i$ contains at least $k/21$ points.
\Lm{bound} shows that the probability $p_{A}$ that we have such a
square anywhere in $S$ is bounded by
\begin{align*}
p_{A}&\le (MN)^2\Prb(\Po(k/N^2)\ge k/21)\\
&\le (MN)^2e^{k/N^2(N^2/21-1-(N^2/21)\log N^2/21)}\\
&<(MN)^2e^{k(1-\log N^2/21)/21}=o(e^{-9k}),
\end{align*}
as long as $N>(21e^{190})^{1/2}$.

\noindent{\bf Type B.} We consider the set $\Sigma$ of circles whose centres are centres
of small squares and which pass through at least one other centre of a small square of our
tiling. Clearly, $|\Sigma|\le (MN)^4$. For each $\Gamma\in\Sigma$, let $R_{\Gamma}$ be the
set of squares $S_i$ that lie entirely within distance $\frac52\ell\sqrt 2$ of~$\Gamma$,
where $\ell=\sqrt k/N$ is the side length of the small squares. Type B configurations are those
for which, for some $\Gamma\in\Sigma$,
\begin{equation}\label{TypeB}
 \frac{k}{N^2}\sum_{S_i\in R_{\Gamma}}d(S_i)\ge \frac{\eps k}{2}.
\end{equation}
Write $c(\Gamma)$ and $r(\Gamma)$ for the centre and radius of $\Gamma$, and let
$\Gamma^t$ be the circle with centre $c(\Gamma)$ and radius $r(\Gamma)+t$. Then
since $|S\cap\Gamma^t|\le |\partial S|=4M\sqrt{k}$ for all $t\ge -r(\Gamma)$, we see that
the area $|R_{\Gamma}|$ of each $R_{\Gamma}$ is at most
\[
  |R_{\Gamma}|\le\int_{-5\ell/\sqrt 2}^{+5\ell/\sqrt 2}(S\cap\Gamma^t)\ dt
  \le\int_{-5\ell/\sqrt 2}^{+5\ell/\sqrt 2}4M\sqrt{k}\ dt=(5\ell\sqrt 2)(4M\sqrt{k})<30Mk/N.
\]

Thus each $R_{\Gamma}$ contains at most $30MN$ squares.
Therefore, if \eqref{TypeB} holds for some $R_{\Gamma}$, then that $R_{\Gamma}$
contains at least
\[
 \frac{\eps k}{2}-\frac{30Mk}{N^2}=k\left(\frac{\eps}{2}-\frac{30M}{N^2}\right)
\]
points. Thus for $N\ge N_1(\eps,M)=(180M/\eps)^{1/2}$, the $R_{\Gamma}$ chosen above must contain at least
$\frac{\eps k}{3}$ points. Thus by \Lm{bound} the probability $p_B$ that $\cP_A$ belongs
to a Type B configuration is bounded by
\begin{align*}
p_B &\le (MN)^4\Prb({\rm Po}(30Mk/N)\ge\eps k/3)\\
&\le (MN)^4e^{\frac{30Mk}{N}\left(\frac{\eps N}{90M}-1-\frac{\eps N}{90M}\log\left(
\frac{\eps N}{90M}\right)\right)}\\
&< (MN)^4e^{\frac{\eps k}{3}\left(1-\log\left(\frac{\eps N}{90M}\right)\right)}
=o(e^{-9k}),
\end{align*}
as long as $N\ge N_2(\eps,M)$. We shall also assume $N\ge N_3(\eps,M)=M^2/2\eps$
for the next lemma.

\begin{lemma}\label{l:key}
 Suppose that $F$ is a good configuration, that $S_1$ and $S_2$ are two squares in $S$,
 and that $\cP$ and $\cP'$ are two point sets belonging to $F$.
 If there is no edge in $G_{S,k}(\cP)$ from any vertex in $S_1$
 to any vertex in $S_2$, then there is no edge in $G_{S,k(1-\eps)}(\cP')$
 from any vertex in $S_1$ to any vertex in $S_2$.
 \end{lemma}
\begin{proof}
If either $S_1$ or $S_2$ is empty in $\cP$ then the same square will be empty in $\cP'$,
so that in both cases there will be no edges from $S_1$ to $S_2$. Otherwise, pick
$x_1\in \cP\cap S_1$ and $x_2\in \cP\cap S_2$. Suppose for a contradiction that there are
$y_1\in \cP'\cap S_1$ and $y_2\in \cP'\cap S_2$ such that $y_1y_2\in E(G_{k(1-\eps)}(\cP'))$.
Without loss of generality, $y_2$ is one of the $k(1-\eps)$ nearest neighbours of $y_1$.
Let $z_1$ and $z_2$ be the centre points of $S_1$ and $S_2$ respectively
and let $\ell=\frac{\sqrt k}{N}$ be the side length of the small squares.
Let $d=\|z_1-z_2\|$ be the distance between $z_1$ and $z_2$. Now
$\|z_i-y_i\|\le \frac12\ell\sqrt 2$, and $\|z_i-x_i\|\le\frac12\ell\sqrt 2$,
so
\[
 B(x_1,\|x_2-x_1\|)\subseteq B(x_1,d+\ell\sqrt2)\subseteq B(z_1,d+\tfrac32\ell\sqrt2)
\]
and
\[
 B(y_1,\|y_2-y_1\|)\supseteq B(y_1,d-\ell\sqrt2)\supseteq B(z_1,d-\tfrac32\ell\sqrt2)
\]
where $B(x,r)$ denotes the disc or radius $r$ about the point $x$.
Now, every square that meets $B(z_1,d-\tfrac52\ell\sqrt2)$ lies inside
$B(z_1,d-\tfrac32\ell\sqrt2)$, and every square that meets
$B(z_1,d+\tfrac32\ell\sqrt2)$ lies inside $B(z_1,d+\tfrac52\ell\sqrt2)$.
Let $R_0$ be the union of the squares meeting
$B(z_1,d-\tfrac52\ell\sqrt2)$ and let $\Gamma\in\Sigma$ be the circle
through $z_2$ centred at $z_1$. Recall that $R_\Gamma$ consists of all the squares
strictly contained in $B(z_1,d+\tfrac52\ell\sqrt2)\setminus B(z_1,d-\tfrac52\ell\sqrt2)$.
Therefore
\[
 R_0\subseteq B(y_1,\|y_2-y_1\|)\qquad\text{and}\qquad
 B(x_1,\|x_2-x_1\|)\subseteq R_0\cup R_\Gamma.
\]
But $B(y_1,\|y_2-y_1\|)$ (and hence $R_0$) contains at most $k(1-\eps)$
points of $\cP'$ and $R_\Gamma$ contains at most $\eps k/2$ points
of $\cP'$,
since $F$ is not of Type B.
Thus $R_0\cup R_\Gamma$ contains at most $k(1-\eps/2)$
points of $\cP'$. Since no square has $d(S_i)=\infty$
(because $F$ is not of Type A), this implies
$R_0\cup R_\Gamma$ (and hence $B(x_1,\|x_2-x_1\|)$) contains at most
\[
 k(1-\eps/2)+(1/N)|R_0\cup R_\Gamma|\le k(1-\eps/2)+(1/N)M^2 k<k
\]
points of $\cP$. Thus $x_2$ is one of the $k$ nearest neighbours
of $x_1$ in $G_{S,k}(\cP)$, contradicting the assumption that
$G_{S,k}(\cP)$ contains no edge between $S_1$ and $S_2$.
\end{proof}

Let $\cF$ be a set of configurations. Write $I(\cF)$ for the event that
$\cP$ belongs to some $F\in\cF$. Also, let $\cG$ be the set of good
configurations.

\begin{figure}
\centerline{\includegraphics{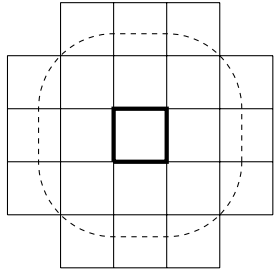}}
\caption{Any two points in the centre square are joined in $G_{S,k}$ provided
there are not more than $k$ points in the union of the 21 squares shown.}
\label{f:3}
\end{figure}

\begin{lemma}\label{l:engine}
There is a subset $\cY\subseteq\cG$ of configurations such that
\[
 A_k\cap I(\cG)\subseteq I(\cY)
 \subseteq A'_{k(1-\eps)}\cap I(\cG).
\]
\end{lemma}
\begin{proof}
Set
\[
 \cY=\{F\in\cG: A_k\cap I(\{F\})\ne\emptyset\},
\]
so that
\[
 A_k\cap I(\cG)\subseteq I(\cY)
\]
automatically holds. Suppose that $\cP$ belongs to a good configuration $F$. If $A_k$
occurs then $A'_{k(1-\eps)}$ occurs for every $\cP'$ belonging to the same $F$.
For suppose that $\cP$ is a point set for which $A_k$ occurs, and let $T$ be the set of
squares of $S$ containing a point of the component $C$ lying within $S'$. Since $F$
is not of Type~A, there are less than $k$ points within distance $\ell\sqrt2=\sqrt{2k}/N$ of any
point of $\cP$, and hence any point of $\cP$ in any square of our tiling is connected
to all other points of $\cP$ in the same square (see \Fg{3}).
Hence there is no edge in $G_{S,k}(\cP)$ from
any square of $T$ to any square of $S\setminus T$. By \Lm{key}, for any $\cP'$ belonging to $F$
there is thus no edge in $G_{k(1-\eps)}(\cP')$ from any square of $T$ to any square of
$S\setminus T$. Therefore, there is some component contained in $T$ in $G_{k(1-\eps)}(\cP')$.
This component lies within the enlarged central region $S''$ for the event $A'_{k(1-\eps)}$,
since $\tfrac34\sqrt{k(1-\eps)}>\tfrac12\sqrt k+\ell\sqrt2$ for $\eps<\tfrac12$ and large $N$.
Therefore, $A'_{k(1-\eps)}$ occurs for any $\cP'$ belonging to~$F$.
\end{proof}

\begin{lemma}\label{l:pconf}
 For any good configuration $F$, $\Prb(I(\{F\}))=e^{-(\theta_F+o(1))k}$ as $k\to\infty$,
 where $\theta_F$ is some constant depending on~$F$.
\end{lemma}
\begin{proof}
By \Lm{plogp} the probability of there being exactly $\rho_i(k/N^2)$ points
in each square $S_i$ is
\[
 \exp\left\{\sum (\rho_i-1-\rho_i\log\rho_i)|S_i|+O((MN)^2\log((MN)^2k))\right\}
\]
where we have used the fact that $\rho_i(k/N^2)<k$. To calculate the probability
of the configuration $F$ occurring, we sum over all possible values of each
$\rho_i$ consistent with the specified value of $d(S_i)$. Since there are
at most $N^2k$ values of $\rho_i$ for each $i$, we get
\[
 \Prb(I(\{F\})=\exp\left\{\sum (\tilde\rho_i-1-\tilde\rho_i\log\tilde\rho_i)|S_i|
 +O((MN)^2\log((MN)^2k\cdot N^2k))\right\}
\]
where $\tilde\rho_i$ is the value of $\rho_i$ that maximizes
$\rho_i-1-\rho_i\log\rho_i$. (The sum is at least the maximum, and at
most the number of terms $(N^2k)^{(MN)^2}$ times the maximum).
Now let $\rho'_i$ be the real number that maximizes
$\rho_i-1-\rho_i\log\rho_i$ in the range of densities consistent
with $d(S_i)$ for any $k$, so $\rho'_i=d(S_i)$ when $d(S_i)\le1$
and $d(S_i)-1/N$ when $d(S_i)>1$. Now
$|\rho_i-\tilde\rho_i|\le N^2/k$ which tends to 0 as $k\to\infty$.
Thus the difference between
$\tilde\rho_i-1-\tilde\rho_i\log\tilde\rho_i$ and
$\rho'_i-1-\rho'_i\log\rho'_i$ is $o_k(1)$. Hence
\[
 \Prb(I(\{F\})=\exp\left\{\sum (\rho'_i-1-\rho'_i\log\rho'_i)|S_i|
 +o(M^2k)\right\}
\]
Setting $\theta_F=-\sum (\rho'_i-1-\rho'_i\log\rho'_i)(1/N^2)$
gives the result.
\end{proof}

\Lm{pconf} implies
\[
 \Prb(I(\cY))=e^{-(\theta+o(1))k}
\]
where
\[
 \theta=\min_{F\in\cY} \theta_F,
\]
since, loosely speaking, the sum of a finite number of (essentially)
exponential functions is (essentially) equal to the one among them
with the least decay rate. Therefore, by \Lm{limsup}, \Lm{tech} and
\Lm{engine},
\[
(4+o(1))p_1(k(1-\eps))\ge e^{-(\theta+o(1))k}\ge p_1(k)-o(e^{-9k})=p_1(k)(1-o(1)).
\]
Finally,
\[
\limsup_{k\to\infty}f_1(k)
=\limsup_{k\to\infty}-\frac{\log((4+o(1))p_1(k(1-\eps)))}{k(1-\eps)}
\le \frac{\theta k}{k(1-\eps)}
=\frac{\theta}{1-\eps},
\]
and
\[
\liminf_{k\to\infty}f_1(k)
=\liminf_{k\to\infty}-\frac{\log(p_1(k))}{k}
\ge \frac{\theta k}{k}
=\theta,
\]
By letting $\eps\to0$ we see that $f_1(k)$ converges to a limit $c_1$.

Now we turn to $c_2$.
We may reuse the same configurations and good configurations to obtain a version of
\Lm{engine} (with an almost identical proof) with $A_k$ and $A'_{k(1-\eps)}$ replaced
by $B_k$ and $B'_{k(1-\eps)}$ respectively. \Lm{limsup}, \Lm{tech} and \Lm{engine}
now give, for some $\theta'=\theta'(\eps)$,
\[
(2+o(1))(p_1(k(1-\eps))+p_2(k(1-\eps)))\ge e^{-(\theta'+o(1))k}\ge p_2(k)-o(e^{-9k})=p_2(k)(1-o(1)).
\]
Hence
\[
(4+o(1))\max\{p_1(k(1-\eps)),p_2(k(1-\eps))\}\ge p_2(k)(1-o(1)),
\]
and so
\[
\limsup_{k\to\infty}\min\{f_1(k),f_2(k)\}\le\min\left\{\frac{\theta'}{1-\eps},c_1\right\},
\]
and
\[
\liminf_{k\to\infty}\min\{f_1(k),f_2(k)\}\ge\min\{\theta',c_1\}.
\]
By letting $\eps\to0$ we see that $\min\{f_1(k),f_2(k)\}$ converges to a limit $c_2$.

\section{Proof of main theorem}

Write $c_{\text{crit}}=\max\{\frac{1}{c_1},\frac{1}{2c_2}\}$.

\begin{theorem}\label{t:final}
 If\/ $c<c_{\rm{crit}}$ and\/ $k=\lfloor c\log n\rfloor$
 then $\Prb(G_{n,k}$ is connected$)\to0$ as $n\to\infty$.
 If\/ $c>c_{\rm{crit}}$ and\/ $k=\lfloor c\log n\rfloor$
 then $\Prb(G_{n,k}$ is connected$)\to1$ as $n\to\infty$.
\end{theorem}
\begin{proof}
We prove the lower bound first. Suppose that $c<c_{\text{crit}}$ and
$k=\lfloor c\log n\rfloor$.
We place $\Theta(n/\log n)$ disjoint squares $S$ (of area $M^2k$) in the interior of
$S_n$, and we place $\Theta(\sqrt{n/\log n})$ disjoint squares $R$ (also of area
$M^2k$) along the boundary of $S_n$, with the squares $R'$ lying along
the boundary of $S_n$. Let $\cP$ be a Poisson process of intensity one in $S_n$, and
consider the restriction of $\cP$ to one of the squares $S_1$. With probability $e^{-(c_1+o(1))k}$,
$S_1$ now contains a small component near its centre, and, by choice of $M$, such a component
would almost certainly remain a component in $G_{n,k}$. The probability that none of the squares $S$
contains a small component (in the respective restricted graph) near its centre is
\begin{align*}
p_{{\rm fail}}&=(1-e^{-(c_1+o(1))k})^{An/\log n}\\
&<\exp\{-A(n/\log n)e^{-(c_1+o(1))k}\}\\
&\le\exp\{-An^{1-o(1)+(c_1+o(1))c}\}\to 0,
\end{align*}
by independence, if $cc_1<1$.

Note that if $c_1=c_2$, we are done. Suppose then that $c_2<c_1$, and
consider the restriction of $\cP$ to one of the squares $R_1$. With probability
$e^{-(c_2+o(1))k}$, $R_1$ now contains a small component in its region $R_1'$, and, again
by choice of $M$, such a component would remain a component in $G_{n,k}$. The probability
that none of the squares $R$ contains a small component (in the respective restricted graph)
lying in $R'$ is
\begin{align*}
p_{{\rm fail}}&=(1-e^{-(c_2+o(1))k})^{B(n/\log n)^{1/2}}\\
&<\exp(-B(n/\log n)^{1/2}e^{-(c_2+o(1))k})\\
&\le\exp(-Bn^{1/2-o(1)-(c_2+o(1))c})\to 0,
\end{align*}
by independence, as long as $cc_1<1/2$. Hence, if either $cc_1<1$ or $cc_2<1/2$, i.e.,
for $c<c_{\text{crit}}$, $G_{n,k}$ will be asymptotically almost surely disconnected.

For the upper bound, suppose that $c>c_{\text{crit}}$ and that $k=\lfloor c\log n\rfloor$.
For notational simplicity, we assume that $c_2<c_1$.
From the proof of Theorem~13 in~\cite{BBSW}, the probability that $G_{n,k}$ contains
a component of size $O(\sqrt{\log n})$ within distance $O(\sqrt{\log n})$ of a corner
of $S_n$ is $n^{o(1)}3^{-k}$, which tends to 0 as $n\to\infty$.
Suppose then that there exists such a small component $H$
far from a corner. One can tile $S_n$ with $\Theta(n/\log n)$ overlapping squares $S$ and the
boundary of $S_n$ with $\Theta(\sqrt{n/\log n})$ overlapping squares $R$ such that $H$ lies in
one of the regions $S'$ or $R'$ of these tiles. (In the overlapping scheme, the centres of the
$S$-tiles form a lattice with horizontal and vertical spacing $\frac14M\sqrt k$, and the
boundary of the $R$-tiles that contain $0$ lie on the perimeter of $S_n$, at intervals of
$\frac14M\sqrt k$.) Therefore, the probability of such a component $H$ arising is at most the
expected number of tiles for which $A_k$ (for an $S$-tile) or $B_k$ (for an $R$-tile) occurs.
But for $c>c_{\text{crit}}$, this expectation is equal to
\[
 A(n/\log n)e^{-(c_1+o(1))k}+B(n/\log n)^{1/2}e^{-(c_2+o(1))k}=o(1).
\]
Hence $G_{n,k}$ is asymptotically almost surely connected.
\end{proof}


\begin{thebibliography}{99}

\bibitem{BBSW} P. Balister, B. Bollob\'as, A. Sarkar and M. Walters,
{\sl Connectivity of random $k$-nearest neighbour graphs},
Advances in Applied Probability {\bf 37} (2005), 1--24.

\bibitem{XuKu} F. Xue and P.R. Kumar,
{\sl The number of neighbors needed for connectivity of wireless networks},
Wireless Networks {\bf 10} (2004), 169--181.

\end{thebibliography}
\end{document}